\documentclass[12pt, reqno]{amsart}
\usepackage{amsfonts}
\usepackage{bbm}
\usepackage{amscd,amsfonts}
\usepackage{amssymb, eucal, amsfonts, amsmath, xypic, latexsym}
\usepackage{pifont}
\usepackage{mathrsfs,color}
\usepackage{amsthm,indentfirst,bm,fancyhdr,dsfont}
\usepackage{graphicx}
\usepackage[all]{xy}

\usepackage{mathrsfs}
\usepackage{amsmath}
\usepackage{amssymb}
\usepackage{hyperref}

\newtheorem{theorem}{Theorem}[section]
\newtheorem{lemma}[theorem]{Lemma}
\newtheorem{prop}[theorem]{Proposition}

\theoremstyle{definition}
\newtheorem{conj}[theorem]{Conjecture}
\newtheorem{defn}[theorem]{Definition}

\newtheorem{rem}[theorem]{Remark}
\newtheorem{remark}[theorem]{Remark}

\newtheorem{jacon}[theorem]{Jacobian Conjecture}

\numberwithin{equation}{section}

\def\mmm{\mathfrak{m}}

\def\fff{\mathfrak{f}}
\def\qqq{\mathfrak{q}}

\def\frakm{\mathfrak{M}}

\def\calm{\mathcal{M}}

\def\calv{\mathcal{V}}
\def\calj{\mathcal{J}}
\def\calr{\mathcal{R}}

\def\bfb{\mathbf{b}}
\def\bd{\mathbf{d}}
\def\ba{\mathbf{a}}

\def\bJ{\mathbf{J}}

\def\bF{\mathbf{F}}

\def\bt{{\mathbf{t}}}

\def\bbc{\mathbb{C}}
\def\bbf{\mathbb{F}}
\def\bbz{\mathbb{Z}}

\def\bba{{\mathbb{A}}}
\def\bbr{{\mathbb{R}}}

\def\bo{{\bar 1}}
\def\bz{{\bar 0}}

\def\sfd{\textsf{d}}

\def\sfU{\textsf{U}}

\def\spec{\mathsf{Spec}}
\def\Jac{\mathsf{Jac}}

\def\id{\mathsf{id}}

\def\ev{{\textsf{ev}}}

{\vskip-\lastskip\medskip
  \noindent
  }%
{\qed\par\medskip
  }

\newcommand{\co}{\mathcal{O}}
\newcommand{\cv}{\mathcal{V}}
\newcommand{\cw}{\mathcal{W}}

\newcommand{\cvv}{{|\mathcal{V}|}}
\newcommand{\cwv}{|\mathcal{W}|}
\newcommand{\phiv}{|\phi|}

\newcommand{\bbamn}{{\bba^{m|n}}}
\newcommand{\bbamnf}{{\bba^{m|n}_\bbf}}

\def\salg{\mathfrak{salg}_\bbf}

\def\sfx{\textsf{X}}
\def\sfp{\textsf{p}}
\def\scrm{\mathscr{M}}

\begin{document}

\title[On automorphisms of affine spaces and of affine spaces]
{On automorphisms  of affine superspaces}

\author{Bin Shu}

\address{School of Mathematical Sciences, Ministry of Education Key Laboratory of Mathematics and Engineering Applications \& Shanghai Key Laboratory of PMMP, East China Normal University, Shanghai 200241, China} \email{bshu@math.ecnu.edu.cn}


\keywords{Affine superspace, super version of Jacobian conjecture}

\thanks{This work is partially supported by the National Natural Science Foundation of China (Grant Nos. 12071136 and 12271345), supported in part by Science and Technology Commission of Shanghai Municipality (No. 22DZ2229014).}

\begin{abstract} In this note,  we propose a super version of Jacobian conjecture on the automorphisms of affine superspaces over an algebraically closed field $\bbf$ of characteristic $0$, which predicts that for  a homomorphism $\varphi$ of the polynomial superalgebra $\calr:=\bbf[x_1,\ldots,x_m; \xi_1,\ldots,\xi_m]$ over $\bbf$, if $\varphi$ satisfies the super version of Jacobian condition (SJ for short), then $\varphi$ gives rise to an automorphism of the affine superspace $\bba_\bbf^{m|n}$.  We verify  the conjecture if  additionally, the set $\scrm$ of maximal $\bbz_2$-homogeneous ideals of  $\calr$ is assumed to be preserved under $\varphi$. The statement is actually proved in any characteristic, i.e. a homomorphism $\varphi$ gives rise to an automorphism of $\bba_\bbf^{m|n}$ if SJ is satisfied with $\varphi$ and the set $\scrm$ is preserved under $\varphi$ for an algebraically closed field $\bbf$  of any characteristic.
\end{abstract}

\maketitle
\setcounter{tocdepth}{1}\tableofcontents

\section{Preliminaries}
\subsection{Classical Jacobian conjecture}
In the classical Jacobian conjecture, the ground field $\bbf$ is $\bbc$ or $\bbr$. Let $\theta: \bba_\bbf^m\rightarrow \bba_\bbf^m$ be a morphism of affine spaces in the sense of Zariski topology, which sends $(t_1,\ldots,t_m)\in \bba_\bbf^m$ to $(\theta_1(t_1,\ldots,t_m),\ldots, \theta_m(t_1,\ldots,t_m))\in \bba_{\bbf}^m$, and $J(\theta)$ denotes the Jacobian matrix which is $({\partial\theta_i\over \partial t_j})_{i,j=1,\ldots,m}$. Set $\vartheta:=\theta^*$, which is the comorphism of $\theta$.

\begin{jacon}
{
If $\operatorname{det}(J(\theta)) \in \mathbb{F}^{\times}:=\bbf\backslash \{0\}$, then $\theta$ is an automorphism of the affine space $\mathbb A^m_\bbf$, that is to say,   $\vartheta$ is an automorphism of the polynomial ring $\mathbb{F}\left[x_1, \ldots, x_m\right]$ where $x_i$ means the $i$th coordinate function for $i=1, \ldots, m$.
}

\end{jacon}
     So we have the Jacobian matrix $J(\vartheta)=({{\partial\vartheta(x_i)}\over {\partial x_j}})_{m\times m}$ associated with $\vartheta$. Obviously, both $J(\vartheta)$ and $J(\theta)$ have the same meaning. We will not distinguish them.

The reader may refer to \cite{KM} {\sl{etc.}} for the recent progress of Jacobian conjecture.

\subsection{A strong version of Jacobian statement}
Recall that the endomorphism space of the polynomial algebra $\bbf[x_1,\ldots,x_m]$  is  in a one-to-one correspondence with the set of morphisms of the affine space $\bba^m_\bbf$ to itself. The automorphisms of $\bba_m^\bbf$ are completely determined by the corresponding automorphisms of the polynomial algebra $\bbf[x_1,\ldots,x_m]$.
In the following, simply denote by $P_m$ the polynomial algebra $\bbf[x_1,\ldots,x_m]$ over $\bbf$.

\begin{theorem}\label{thm: strong ver Jacobi}  (\cite[Theorem 4.1]{CSY} and  \cite[Corollary 5.2]{CSY})
 Let  $\bbf$ be an algebraically closed field of any characteristic. Suppose  $\vartheta$ is a ring endomorphism of $P_m$. If $\textsf{det}(J(\vartheta))\in \bbf^\times$, and $\vartheta$ preserves the set  $\scrm$ of maximal ideals of $P_m$,  then $\vartheta$ is an automorphism of $P_m$.
\end{theorem}

\section{Supercommutative superalgebras}
We always suppose that $\bbf$ is  an algebraically closed field of characteristic not equal to $2$. Recall that a supercommutative superalgebra $A=A_\bz\oplus A_\bo$ over $\bbf$ satisfies $a b= (-1)^{\sfp(a)\sfp(b)}ba$ and  for $\bbz_2$-homogeneous elements $a,b\in A$. Here and further $\sfp$ always denotes   the parity.
Set $\salg$ to be the category of supercommutative $\bbf$-superalgeras.  Throughout the paper,  by superalgebras we mean the objects in $\salg$.

\subsection{} For a superalgebra $A=A_\bz\oplus A_\bo\in\salg$, let $J_A$ denote the ideal generated by $A_\bo$. Set $A^\ev=A\slash J_A$. Then $A^\ev$ becomes an ordinary commutative algebra. Note that $A_\bz$ and $A^\ev$ differ only by nilpotent elements. Hence the prime spectrums $\spec(A_\bz)$ and $\spec(A^\ev)$ coincide.

Note that for $A=A_\bz\oplus A_\bo\in \salg$, the odd part $A_\bo$ consists of nilpotent elements. So any maximal $\bbz_2$-homogeneous ideal $I$ of $A$ must contain $A_\bo$, and $I_\bz$ must be a maximal ideal of $A_\bz$.  A supercommutative $\bbf$-superalgebra $A$ is called local if it has a unique maximal $\bbz_2$-homogeneous ideal. So for a local superalgebra $A\in\salg$ with the unique maximal homogeneous ideal $\frakm$, $A_\bz$ must be a local ring with the vector-space decomposition $A_\bz=\bbf\oplus \mmm$ where $\mmm$ is the unique maximal ideal of $A_\bz$, and $A=\bbf\oplus \frakm$ with $\mmm=\frakm\cap A_\bz$.

\subsection{} Now we recall the notion of super derivations which will be used later. For $A\in \salg$ and an $A$-supermodule $M$, we can regard $M$ as an $A$-bimodule so that $mx=(-1)^{\sfp(m)\sfp(x)}xm$ for $x\in A_{\sfp(x)}$ and $m\in M_{\sfp(m)}$. With this convention, an $\bbf$-linear map  $D:A\rightarrow M$ satisfying the property $D(a)=0$ for all $a\in\bbf$ is said to be  a $\bbz_2$-homogeneous super derivation from $A$ to $M$ if $D(fg)=D(f)g+(-1)^{\sfp(D)\sfp(f)}fD(g)$ for $\bbz_2$-homogeneous elements $f,g\in A$.
 Denote by $\text{Der}(A,M)$ the vector superspace spanned by $\bbz_2$-homogeneous super derivations from $A$ to $M$.

\subsection{} For $A=A_\bz+A_\bo$, $B=B_\bz+B_\bo\in \salg$,  a homomorphism $\varphi$ from $A$ to $B$ is called $\bbz_2$-homogeneous  if the parities are preserved, i.e.  $\varphi(A_\bz)\subset B_\bz$ and $\varphi(A_\bo)\subset B_\bz$.
\vskip5pt

For simplicity, {\sl{we will simply  write ``homogenous" for ``$\bbz_2$-homogenous"  if the context is clear.}}

\section{Superspaces}
Let us  briefly recall some material on superspaces.  The reader may refer to \cite{CCF} or \cite{Var}, in particular to \cite[Chapter 4]{Man} for further information.
\subsection{Ringed superspaces}Recall that a ringed superspace $\cv$ is a topological space $|\cv|$ endowed with a sheaf of supercommutative  $\bbf$-superalgebras. The sheaf is denoted by $\co_\cv$, called the structure sheaf of $\cv$. This ringed superspace $\cv$ is also denoted by $(\cvv, \co_\cv)$ as usually.
 Set $\cv_{\ev}:=(|\cv|, \co_{\cv,\bz})$ which  is an ordinary ringed space on the topology space $|\cv|$ where $\co_{\cv,\bz}$ stands for the sheaf of supercommutative $\bbf$-algebras on $|\cv|$, i.e. for any open subset $U$ of $\cvv$, $\co_{\cv,\bz}(U):=\co_\cv(U)_\bz$ the even part of the superalgebra $\co_\cv(U)$. Similarly, one has the sheaf $\co_{\cv,\bo}$, defined as $\co_{\cv,\bo}(U)=\co_\cv(U)_\bo$ the odd part of the superalgebra $\co_\cv(U)$, which is a sheaf of $\co_{\cv,\bz}$-modules. 

\subsection{Superspaces}

 \begin{defn} A superspace over $\bbf$ is a ringed superspace $\cv$ satisfying that the stalk $\co_{\calv,x}$ is a local superalgebra for all $x\in\cvv$, and that $\co_{\cv,\bo}$ is a quasi-coherent sheaf of $\co_{\cv,\bz}$-modules.
 \end{defn}

 \subsubsection{Morphisms} Let $\cv=(\cvv,\co_\cv)$ and $\cw=(\cwv,\co_\cw)$ be two superspaces. A morphism $\phi: \cv\rightarrow \cw$ of superspaces  is by definition a pair $\phi=(\phiv, \phi^*)$ where $\phiv$ and $\phi^*$ are as follows
 \begin{itemize}
 \item[(1)] $\phiv:\cvv\rightarrow \cwv$ is a continuous map;

 \item[(2)] $\phi^*:\co_\cw\rightarrow \phiv_*\co_\cv$ is a homogeneous map of sheaves of superalgebras on $\cwv$ where $|\phi|_*$ is the direct image sheaf of $\phiv$ on $\cw$. This means, for $U$ open in $\cwv$ there exists a homogeneous homomorphism $\phi^*_U: \co_\cw(U)\rightarrow \co_\cv(|\phi|^{-1}(U))$ compatible with restrictions, i.e. for any open subset $V$ in $\cwv$ with $U\subset V$, the following diagram is commutative
     \[
\begin{CD}
\co_\cw(V) @>{\phi^*_V}>> \co_\cv(\phiv^{-1}(V))\\
@V{\rho_{UV}}VV @VV{\rho_{\phiv^{-1}(U)\phiv^{-1}(V)}}V  \\
\co_\cw(U) @>>{\phi^*_U}>\co_\cv(\phiv^{-1}(U))
\end{CD}
\]
where the vertical maps are restrictions;

 \item[(3)] the map of local superalgebras $\phi_v^*:\co_{\cw,\phiv(v)}\rightarrow \co_{\cv,v}$ for any $v\in\cvv$ is a local morphism, i.e., sends the $\bbz_2$-homogeneous maximal ideal of $\co_{\cw,\phiv(v)}$ to the $\bbz_2$-homogeneous maximal ideal of $\co_{\cv,v}$.
 \end{itemize}

 \section{Affine superspaces}

  In the following, we will first recall the affine superspace $\bbamnf$ of super dimension $m|n$ over $\bbf$, which is a ringed superspace with $\bbamnf=(|\bbamnf|=\bbf^m, \co_{\bbamnf})$ (see \cite[Chapter 4]{Man} for the further details).

 \subsection{}\label{sec: aff super} Let  $\calr=\bbf[x_1,\ldots,x_m; \xi_1,\ldots,\xi_n]$ be the polynomial superalgebra with even indeterminates $x_1,\ldots,x_m$, and odd indeterminates $\xi_1,\ldots,\xi_n$. This means, $\bbf[x_1,\ldots,x_m]$ is an ordinary polynomial algebra, but the odd indeterminates $\{\xi_k\}_{k=1,\ldots,n}$ generate the exterior algebra $\bigwedge(\xi_1,\ldots,\xi_n)$ with $\xi_k\xi_l=-\xi_l\xi_k$  and $x_i\xi_k=\xi_k x_i$ for $i=1,\ldots,m$ and $k,l=1,\ldots,n$. As a vector space, $\calr=\bbf[x_1,\ldots,x_m]\otimes \bigwedge(\xi_1,\ldots,\xi_n)$. The polynomial superalgebra $\calr$ can be regarded as an exterior algebra generated by $\{\xi_k\mid k=1,\ldots,n\}$ over $\bbf[x_1,\ldots,x_m]$, which can be denoted by $\bigwedge_{\bbf[x_1,\ldots,x_m]}(\xi_1,\ldots,\xi_n)$.
  The $(x_i, \xi_k)_{i=1,\ldots,m; k=1,\ldots,n}$ is called a coordinate system of the affine superspace $\bbamn$. We will denote by $\bJ$  the ideal of $\calr$ generated by $\xi_1,\ldots,\xi_n$.

Then $\calr=\calr_\bz\oplus\calr_\bo$ is a supercommutative superalgebra with
$$\calr_\bz=\{f_0+\sum_{\sharp K \text{ even}}f_K\xi_K\mid K=\{1\leq k_1<\cdots< k_r\leq n\}\}$$
where $\xi_K=\xi_{k_1}\cdots\xi_{k_r}$ with $\sharp K=r$ ($\sharp$ denotes the number of elements of a set), and $f_0,f_K \in\bbf[x_1,\ldots,x_m]$, and
$$\calr_\bo=\{\sum_{\sharp L  \text{ odd}}f_L\xi_L\mid L=\{1\leq l_1<\cdots< l_s\leq n\}\}.$$
Clearly, we can identify $\calr^\ev$ with $\bbf[x_1,\ldots,x_m]$, correspondingly $\calr=\calr^\ev[\xi_1,\ldots,\xi_n]$ which means that $\calr$ is the exterior algebra over $\calr^\ev$ generated by $\xi_1,\ldots,\xi_n$.

Recall that the ordinary affine space $\bba^m_\bbf$ consists of the topology space $\bbf^m$ with the Zariski topology, and the sheaf $\co_{\bba^m}$ of regular functions on $\bbf^m$. We define the sheaf $\co_{\bbamnf}$ of superalgebras on $\bbf^m$ in the following way.
Given an open subset $U$ in $\bbf^m$, $$\co_{\bbamnf}(U)=\co_{\bbf^m}(U)\otimes \bigwedge(\xi_1,\dots,\xi_n)=\bigwedge_{\co_{\bbf^m}(U)} (\xi_1,\dots,\xi_n).$$
Clearly, the stalk $\co_{\bbamnf,\ba}$ at $\ba\in \bbf^m$ is equal to $\co_{\bbf^m,\ba}\otimes \bigwedge(\xi_1,\dots,\xi_n)$; and $\co_{\bbamnf}(U)=\bigcap_{\ba\in U}\co_{\bbf^m,\ba}\otimes \bigwedge(\xi_1,\dots,\xi_n)$ for any open subset $U$ in $\bbf^m$. Consequently, we have

\begin{lemma}
$\calr=\bigcap_{\ba \in\bbf^m}\co_{\bbamnf,\ba}$.
\end{lemma}

In particular, $\co_{\bbamnf,\ba}$ is a local ring with the unique maximal $\bbz_2$-homogeneous ideal $\mmm_{\bbamn,\ba}$ which is equal to $\mmm_\ba+\bJ$ where $\mmm_\ba$ means the maximal ideal of the stalk $\co_{\bbf^m,\ba}$  at $\ba$ of the structure sheaf  for the affine space $\bba_\bbf^m$.

\subsection{Morphisms and automorphisms}
In the subsequent arguments, {\sl {we will simply write $\bbamnf$ as $\bbamn$ if the context is clear.}}

\begin{lemma}\label{lem: 4.2} Suppose $\phi=(\phiv, {\phi}^*)$ is a morphism from $\bbamn$ to itself, then there are $f_1,\ldots, f_m\in\calr_\bz$ and $q_1,\ldots,q_n\in \calr_\bo$ such that
 \begin{align}\label{eq: morTOpoly}
 \phi^*(x_i)=f_i, i=1,\ldots,m; \;\; {\phi}^*(\xi_k)=q_k, k=1,\ldots,n.
 \end{align}
  Such a sequence $(f_i,q_k)_{i=1,\ldots,m;k=1,\ldots,n}$ is unique.
\end{lemma}

\begin{proof} It follows from the definition.
\end{proof}

By the above lemma, we give the following definition.
\begin{defn}\label{def: auto} Let $\phi=(\phiv,\phi^*)$ be a morphism corresponding to the sequence $(f_i,q_k)_{i=1,\ldots,m;k=1,\ldots,n}$ as the above lemma. Call $\phi$  an automorphism of the affine superspace $\bbamn$ if $(f_i,q_k)_{i=1,\ldots,m;k=1,\ldots,n}$ forms a coordinate system of $\bbamn$.
\end{defn}

\begin{prop}\label{prop: 4.4} Suppose  $\phi=(\phiv,\phi^*)$ is a morphism from $\bbamn$ to itself.
\begin{itemize}
\item[(1)]
If  $\phi$ is  an automorphism then the following statements hold.
\begin{itemize}
\item[(1.1)] $\phiv$ is an automorphism of the affine space $\bbf^m$.

\item[(1.2)] $\phi^*_{\bbf^m}$ is a homogeneous automorphism of the superalgebra $\calr$. 
\end{itemize}
\item[(2)] Conversely, if Statement (1.2) holds, then $\phi$ is an automorphism of $\bbamn$.
    \end{itemize}
\end{prop}

\begin{proof} For the part (1),  we suppose $\phi=(\phiv,\phi^*)$ is an automorphism. As to Statement (1.1), from the definition of automorphisms it follows that $f_i^\sharp$, $i=1,\ldots,m$, are algebraically-independent  generators of $\bbf[x_1,\ldots,x_m]$ where $f_i^\sharp$ is the image of $f_i$ in $\calr^\ev$ for $i=1,\ldots,m$, and we identify $\calr^\ev$ with $\bbf[f_1^\sharp,\ldots,f_m^\sharp]=\bbf[x_1,\ldots,x_m]$. Hence $\phi^*_{\bbf^m}|_{\bbf[x_1,\ldots,x_m]}$ becomes an endomorphism of $\bbf[x_1,\ldots,x_m]$, which admits the inverse $\psi$ mapping $f_i^\sharp$ to $x_i$. Hence $\phi^*_{\bbf^m}|_{\bbf[x_1,\ldots,x_m]}$ gives rise to an automorphism of the affine space $\bba_\bbf^m$ which is just $\phiv$.

As to Statement (1.2), by definition  $(f_i,q_k)_{i=1,\ldots,m;k=1,\ldots,n}$ forms a coordinate system of $\bbamn$, which means $f_i,q_k$, $i=1,\ldots,m; k=1,\ldots,n$ generate $\calr$ via the same defining relations as the ones of $\{x_i,\xi_k\}$. So  $\phi^*_{\bbf^m}$ is a homogeneous  automorphism of the superalgebra $\calr$. This completes the proof of the necessary part.

For Part (2), suppose Statement (1.2) already holds. Then $f_1,\ldots,f_m$ are algebraically independent even generators, and $q_1,\ldots,q_n$ are linearly independent odd generators of $\calr$. Furthermore, $\calr=\bbf[f_1,\ldots,f_m]\otimes \bigwedge(q_1,\ldots,q_n)$.  Hence $f_1,\ldots,f_m; q_1,\ldots, q_n$ constitute a coordinate system of $\bbamn$.
The proof is completed.
\end{proof}

\section{Automorphisms of the affine superspaces and super version of Jacobian conjecture}

 Let $\phi=(\phiv, {\phi}^*)$ be a morphism from $\bbamn$ to itself. By Lemma \ref{lem: 4.2}, there are $f_1,\ldots, f_m\in\calr_\bz$ and $q_1,\ldots,q_n\in \calr_\bo$ satisfying (\ref{eq: morTOpoly}). Set
$$\Jac({{\partial f_i}\over {\partial x_j}}):=\det({{\partial f_i}\over {\partial x_j}})_{i,j=1,\ldots,m}\in \calr$$
and
$$\Jac({{\partial q_k}\over {\partial \xi_l}}):=\det({{\partial q_k}\over {\partial \xi_l}})_{k,l=1,\ldots,n}\in \calr$$
where $({{\partial f_i}\over {\partial x_j}})_{i,j=1,\ldots,m}$ and $({{\partial q_k}\over {\partial \xi_l}})_{k,l=1,\ldots,n}$ stand for the matrices of  size  $m\times m$ and of size $n\times n$, respectively, with entries in $\calr$.

\subsection{} We first have the following result.

\begin{prop}\label{prop: 5.1} Let $\phi$ be a morphism from $\bbamn$ to itself. If $\phi$ is an automorphism, then  $\Jac({{\partial f_i}\over {\partial x_j}})\in \bbf^\times+\bJ$ and $\Jac({{\partial q_k}\over {\partial \xi_l}})\in \bbf^\times +\bJ$ where $\bbf^\times=\bbf\backslash\{0\}$.
\end{prop}

Before the proof, let us recall the tangent space of $\bbamn$ at a point $\bt\in \bbf^m$. In the following, we simply denote $\bbamn$ by $\sfx$.
By definition, the tangent space of $\sfx$ at $\bt$ is $T_\bt(\sfx)=\text{Der}(\co_{\sfx,\bt}, \bbf)$, where $\bbf$ is regarded as an $\co_{\sfx,\bt}$-modules  via the identification $\bbf\cong \co_{\sfx,\bt}\slash \mmm_{\sfx,\bt}$ (see \cite[7.9]{AM}. Keep it in mind that $\bbf$ is an algebraically closed field), where $\mmm_{\sfx,\bt}$ is the unique maximal $\bbz_2$-homogeneous ideal in $\co_{\sfx,\bt}$ as in \S\ref{sec: aff super}. The tangent space $T_\bt(\sfx)$ has  basis $({\partial\over{\partial x_i}})_\bt$,  $({\partial\over{\partial\xi_k}})_\bt$, $i=1,\ldots,m$, $k=1,\ldots,n$.

\vskip5pt
Now we give the proof of Proposition \ref{prop: 5.1}.
\vskip-5pt
\begin{proof} The automorphism $\phi$ of $\sfx$ leads to the homogeneous linear isomorphism  $\sfd\phi_\ba$ of the tangent spaces from $T_\ba(\sfx)$ onto $T_{|\phi|(\ba)}(\sfx)$ for any $\ba\in \bbf^m$, where  $\sfd\phi_\ba$ sends  $D\in T_\ba(\sfx)$ to $D\circ \phi_{\phiv(\ba)}^*\in T_{|\phi|(\ba)}(\sfx)$.
By definition, $\{f_i, q_k\}_{i=1,\ldots,m,k=1,\ldots,q}$ becomes a new coordinate system of $\bbamn$. Hence $T_{|\phi|(\ba)}(\sfx)$ has basis
$({\partial\over{\partial f_i}})_{|\phi|(\ba)}$,  $({\partial\over{\partial q_k}})_{|\phi|(\ba)}$, $i=1,\ldots,m$, $k=1,\ldots,n$. The linear homogeneous isomorphism  $\sfd\phi_\ba$ for any $\ba\in \bbf^m$ yields that for any $\bt\in \bbf^m$ there is a basis transform in $T_\bt(\sfx)$ as below
\begin{align}
({\partial\over{\partial x_i}})_\bt&=\sum_{j=1}^m({\partial f_j\over{\partial x_i}})_\bt({\partial\over{\partial f_j}})_\bt,\cr
({\partial\over{\partial \xi_k}})_\bt&=\sum_{l=1}^n({\partial q_l\over{\partial \xi_k}})_\bt({\partial\over{\partial q_l}})_\bt.
\end{align}
Hence both matrices $({\partial f_j\over{\partial x_i}})_{i,j=1,\ldots,m}$, and $({\partial q_l\over{\partial  \xi_k}})_{k,l=1,\ldots,n}$ are invertible at $\bt$ for any $\bt\in\bbf^m$.

If either  $\Jac({{\partial f_i}\over {\partial x_j}})\notin \bbf^\times+\bJ$ or
$\Jac({{\partial q_k}\over {\partial \xi_l}})\notin \bbf^\times+\bJ$, we show that it gives rise to  contradictions.
Suppose $\Jac({{\partial f_i}\over {\partial x_j}})\notin \bbf^\times+\bJ$, then we can write $\Jac({{\partial f_i}\over {\partial x_j}})= a+ f+ \eta$ where $a\in \bbf$, $f\in \sum_{i=1}^m\bbf[x_1,\ldots,x_m]x_i$, and $\eta\in \bJ$ with $a+f\notin\bbf^\times$. Then the zero points of $a+f$ forms an empty closed subset $\bF$ of $\bbf_m$. Hence there exists $\ba\in \bF$ such that $\Jac({{\partial f_i}\over {\partial x_j}})(\ba)=\eta(\ba)$ which is a nilpotent element. Hence $\Jac({{\partial f_i}\over {\partial x_j}})(\ba)$ is not invertible. This contradicts the linear isomorphism  $\sfd\phi_\ba$ of the tangent spaces from $T(\bbamn)_\ba$ onto $T(\bbamn)_{|\phi|(\ba)}$ for any $\ba\in \bbf^m$.

Similarly, we can show a contradiction whenever $\Jac({{\partial q_k}\over {\partial \xi_l}})\notin \bbf^\times+\bJ$.
This completes the proof.
\end{proof}

\subsection{Conjecture}
\begin{conj} Suppose $\bbf$ is $\bbr$ or an algebraically closed field of characteristic $0$.
Consider the above morphism $\phi$ from the affine superspace $\bbamn$ over $\bbf$ to itself. If  the following super-Jacobian condition:
\begin{align}\label{eq: Jac cond}
&\Jac({{\partial f_i}\over {\partial x_j}})\in \bbf^\times+\bJ;\cr
&\Jac({{\partial q_k}\over {\partial \xi_l}})\in \bbf^\times +\bJ
\end{align}
is satisfied,
 then $\phi$ is an automorphism.
\end{conj}


\begin{remark}\label{rem: 5.3}
(1) Affine spaces can  be regarded special cases of affine superspaces, i.e. $\bba_\bbf^m$ can be regarded as $\bbamn$ with $n=0$. The super version of Jacobian conjecture is indeed an extension of the ordinary Jacobian conjecture.

(2) Consider the canonical homomorphism $\pi:\calr\rightarrow \calr^\ev=\calr\slash \bJ$. The image $\pi(\Jac({{\partial f_i}\over {\partial x_j}}))$ is exactly $\det({{\partial \pi(f_i)}\over {\partial x_j}})_{i,j=1,\ldots,m}$.
When the super Jacobian condition (\ref{eq: Jac cond}) is satisfied,    $\det(J(|\phi|^*))=\det({{\partial \pi(f_i)}\over {\partial x_j}})_{i,j=1,\ldots,m}\in \bbf^\times$.   This yields that $|\phi|$  meets the ordinary Jacobian condition.
\end{remark}

\begin{lemma}\label{lem: jacobi-1} Let $\bbf$ be an field of characteristic $0$, or an algebraically closed field of characteristic $p>0$.
Let $\phi$ be a morphism from $\bbamn$ to itself. If the super-Jacobian condition (\ref{eq: Jac cond}) is satisfied and $|\phi|$ is an automorphism of $\bbf^m$,
then $f_i=\fff_i\mod\bJ$, and
$q_k=\sum_{l=1}^n c_{kl}\xi_j\mod \bJ^2$ where  $\fff_i\in\bbf[x_1,\ldots,x_m]$ satisfying $|\phi|^*(x_i)=\fff_i$ $(i=1,\ldots,m)$, and
 $c_{kl}\in\bbf$ with 
$$0\ne\det(c_{kl})_{k,l=1,\ldots,n}\equiv\Jac({{\partial q_k}\over {\partial \xi_l}})\mod\bJ.$$
\end{lemma}

\begin{proof} It directly follows from the definitions.
\end{proof}

\subsection{} In the following, we suppose $\bbf$ is an algebraically closed field of any  characteristic.  The affine space $\bba_\bbf^{m|n}$ is simply denoted by $\bbamn$.
\begin{theorem} Let $\phi$ be a morphism from $\bbamn$ to itself. If the super-Jacobian condition (\ref{eq: Jac cond}) is  satisfied,
and additionally $\phi^*_{\bbf^m}$ preserves the set of maximal $\bbz_2$-homogeneous ideals of $\calr$, then $\phi$ is an automorphism of $\bbamn$.
\end{theorem}

\begin{proof} Recall that for $\calr=\calr_\bz\oplus\calr_\bo$, any maximal $\bbz_2$-homogeneous ideal $I$ of $\calr$ must be of the form $I_\bz\oplus \calr_\bo$ where $I_\bz$ is a maximal ideal of $\calr_\bz$. Note that $\calr_\bz=\bbf[x_1,\ldots,x_m]+\calr_\bz\cap \sum_{1\leq k<l\leq n}\calr \xi_k\xi_l$, and the second part $\calr_\bz\cap \sum_{1\leq k<l\leq n}\calr \xi_k\xi_l$ is nilpotent. Hence the maximal ideal $I$ can be further written as
$$\sum_{i=1}^m \bbf[x_1,\ldots,x_m](x_i-a_i)+\sum_{1\leq k<l\leq n}\calr \xi_k\xi_l+\calr_\bo$$
for some point $\ba=(a_1,\ldots,a_m)\in \bbf^m$. Denote $\frakm_\ba:=\sum_{i=1}^m \bbf[x_1,\ldots,x_m](x_i-a_i)$. Then \{$\frakm_\ba\mid \ba\in \bbf^m\}$ forms  the complete set of maximal ideals of  $\bbf[x_1,\ldots,x_m]\cong \calr^\ev$. By the assumption, $\phi^*_{\bbf^m}$ preserves the set of maximal $\bbz_2$-homogeneous ideals of $\calr$, consequently $\pi\circ \phi^*_{\bbf^m}$ preserves the set of maximal ideals of $\calr^\ev$.
Hence the comorphism $\phiv^*$ preserves the set of all maximal ideals of $\bbf[x_1,\ldots,x_m]$. By Remark \ref{rem: 5.3}(2) and Theorem \ref{thm: strong ver Jacobi},
$\phiv$ is an automorphism of the affine space $\bba_\bbf^m$. Hence the comorphism $\phiv^*$ is an automorphism of $\bbf[x_1,\ldots,x_m]$. According to Proposition \ref{prop: 4.4} and its proof, $f_i^\sharp, i=1,\ldots,m$ are algebraically-independent generators of $\bbf[x_1,\ldots,x_m]$.

In the following,  we will show that  $\phi^*_{\bbf^m}$ is an   automorphism of the superalgebra $\calr$. By Lemma \ref{lem: jacobi-1}, $f_i=\fff_i \mod\bJ$, and $q_k=\sum_{l=1}^n c_{kl}\xi_j\mod \bJ^2$ where we can identify  $f_i^\sharp$ with the element $\fff_i\in \bbf[x_1,\ldots,x_m]$ which satisfies $|\phi|^*(x_i)=\fff_i$, and
$c_{kl}\in\bbf$ with
\begin{align}\label{eq: jacobi}
0\ne\det(c_{kl})_{k,l=1,\ldots,n}\equiv\Jac({{\partial q_k}\over {\partial \xi_l}})\mod\bJ.
\end{align}
By the arguments in the above paragraph  $\frak{f}_i\in\calr_\bz$ are algebraically independent for $i=1,\ldots,m$. Similarly, by (\ref{eq: jacobi}) we have that  $\frak{q}_k:=\sum_{l=1}^n c_{kl}\xi_j\in\calr_\bo$ are $\bbf$-linearly independent for $k=1,\ldots,n$.  Consequently, $\calr=\bigwedge_{\bbf[\fff_1,\ldots,\fff_m]}(\qqq_1,\ldots,\qqq_n)$.
Then we can define a homomorphism $\vartheta$ of $\calr$ via sending $x_i$ to $\fff_i$ and $\xi_k$ to $\qqq_k$ respectively, for $i=1,\ldots,m$ and $k=1,\ldots,n$. By definition, $\vartheta$ is an automorphism of $\calr$ with $\vartheta|_{\bbf[x_1,\ldots,x_m]}=|\phi|^*$. Hence $\vartheta^{-1}$ is also an automorphism of $\calr$ defined by sending the ordinates $(\fff_1,\ldots,\fff_m; \qqq_1,\ldots,\qqq_n)$ to
$(x_1,\ldots,x_m;\xi_1,\ldots,\xi_n)$.

Now we consider  the endomorphism $\upsilon:=\vartheta^{-1}\circ\phi^*_{\bbf^m}$ of $\calr$. It's easily seen that $\upsilon=\id_\calr+ \nu$ with $\nu$ being nilpotent.  This is because $\nu=\upsilon-\id_\calr$ satisfies that $\nu(x_i)\in \bJ$ and $\nu(\xi_k)\in \bJ^2$. Hence $\nu^{n+1}=0$. This yields that
$\upsilon^{n+1}=\id_\calr$. Hence $\phi^*_{\bbf^m}$ is an automorphism of $\calr$. By Proposition \ref{prop: 4.4}(2), $\phi$ is an automorphism of $\bbamn$.
\end{proof}

\vskip20pt
\subsection*{Acknowledgements} The author expresses his thanks to the anonymous referee for his or her helpful comments.

\end{document}